\documentclass[10pt]{amsart}
\usepackage[margin=1.2in]{geometry}
\usepackage{amsmath, amsthm, amssymb, hyperref, bm,verbatim,mathrsfs}
\usepackage{stmaryrd,enumerate}
\usepackage[numbers, sort&compress]{natbib}
\theoremstyle{plain}
\newtheorem{thrm}{Theorem}    
\newtheorem{lmm}{Lemma}

\newcommand{\Mod}[1]{\ (\mathrm{mod}\ #1)}

\begin{document}
\title{Numbers in a Beatty sequence which are orders only of cyclic, abelian or nilpotent groups}
\author{Kang Shengyu}
\address{College of Mathematics and Statistics, Chongqing University, Chongqing, 401331, PR China}
\email{kangshengyu@stu.cqu.edu.cn}
\hfuzz=5pt
\begin{abstract}
 Let \(C(x)\), \(A(x)\), and \(N(x)\) denote the counting functions of cyclic, abelian, and nilpotent numbers not exceeding \(x\), respectively. Their asymptotic formulas have been established in recent work by Pollack and Just. In this paper, by adapting the methods of Pollack and Just, we study the distribution of these numbers in Beatty sequences \(\mathcal{B}_{\alpha,\beta} = ([\alpha n + \beta])_{n=1}^{\infty}\), where \(\alpha > 1\) is an irrational number of finite type and \(\beta\) is a fixed real number. We prove that the counting functions \(\#C^*(x)\), \(\#A^*(x)\), and \(\#N^*(x)\) for cyclic, abelian, and nilpotent numbers in Beatty sequences satisfy asymptotic formulas that differ from those of Pollack and Just only by a factor \(1/\alpha\). 
\end{abstract}

\maketitle
\markboth{Cyclic, Abelian and nilpotent numbers}{Kang Shengyu}
\section{Introduction}
\par A positive integer $n$ is called cyclic if every group of order $n$ is cyclic. It is called abelian if every group of order $n$ is abelian, and nilpotent if every group of order $n$ is nilpotent. We denote by $C(x)$ the number of cyclic numbers not exceeding $x$, $A(x)$ the number of abelian numbers not exceeding $x$, and $N(x)$ the number of nilpotent numbers not exceeding $x$. Pollack \cite{Pollack2021} and Matthew \cite{Just2023} have shown the asymptotic formulas for cyclic, abelian, and nilpotent numbers, namely: 
\begin{align*}
    C(x) & = \frac{e^{-\gamma}x}{\log_3{x}} \left( 1- \frac{\gamma}{\log_3{x}}+\frac{\gamma^2+\frac{1}{12\pi^2}}{(\log_3{x})^2}- \frac{\gamma^3+\frac{\gamma\pi^2}{4}+\frac{2\zeta(3)}{3}}{(\log_3{x})^3}+ \ldots \right), \\
    A(x)-C(x) & = \frac{e^{-\gamma}x}{(\log_3{x})^2\log_2{x}} \left( 1- \frac{2\gamma}{\log_3{x}}+\frac{3\gamma^2+\frac{1}{4\pi^2}}{(\log_3{x})^2}- \frac{4\gamma^3+{\gamma\pi^2}+\frac{8\zeta(3)}{3}}{(\log_3{x})^3}+ \ldots \right), \\
    N(x)-A(x) & = \frac{e^{-\gamma}x}{(\log_3{x})^2(\log_2{x})^2} \left( 1+ \frac{1-2\gamma}{\log_3{x}}+\frac{-2\gamma+\frac{5\gamma^2}{2}+\frac{\pi^2}{6}}{(\log_3{x})^2}+ \ldots \right),
\end{align*}
where $\log_k{x}$ is the kth iterate of the natural logarithm, $\gamma$ is the Euler constant.
 \par We will study cyclic (abelian, nilpotent) numbers in Beatty sequences. For fixed real numbers $\alpha \text{ and } \beta$, the non-homogeneous Beatty sequence is defined as\ \  $\mathcal{B}_{\alpha,\beta}:=([ \alpha n+\beta])_{n=1}^\infty$, here for a real number $\theta$,  $[\theta]$ denotes its integer part. For the given irrational number $\alpha$, we assume that the type $\tau(\alpha)=\tau$ is finite, where 
\begin{align*}
    \tau=\sup \left\{ t \in \mathbb{R} : \liminf_{n \to \infty} n^t\|\alpha n\| = 0 \right\},
\end{align*} $\|\cdot\|$ denotes the distance to the nearest integer. For convenience, in the following we assume that $\alpha>1$ is a fixed irrational number of finite type $\tau$ and $\beta$ is a fixed real number.

\par Following previous work \cite{Szele1947} and \cite{Pazderski1959}, $n$ is cyclic precisely when $(n,\varphi(n)) = 1$, where $\varphi$ is Euler's totient function, $(n,\varphi(n))$ is the greatest common divisor of $n$ and $\varphi(n)$. Furthermore, we can define the multiplicative function $\phi(n)$ whose value on a prime power is $\phi(p^a)=(p^a-1)(p^{a-1}-1)\dots(p-1)$. Then a number $n$ is abelian if and only if $n$ is cubefree and $(n,\phi(n)) = 1$, and $n$ is nilpotent if and only if $(n,\phi(n)) = 1$. For convenience, we denote the sets 
\begin{align*}
C_*(x) &= \{ n \le x : (n, \varphi(n)) = 1 \}, &
C^*(x) &= C_*(x) \cap \mathcal{B}_{\alpha,\beta}, \\[6pt]
A_*(x) &= \{ n \le x : (n, \phi(n)) = 1,\ n \text{ is cubefree} \}, &
A^*(x) &= A_*(x) \cap \mathcal{B}_{\alpha,\beta}, \\[6pt]
N_*(x) &= \{ n \le x : (n, \phi(n)) = 1 \}, &
N^*(x) &= N_*(x) \cap \mathcal{B}_{\alpha,\beta},
\end{align*}
so $C(x)=\#C_*(x)$, $A(x)=\#A_*(x)$, $N(x)=\#N_*(x)$. 

\par In Theorem~\ref{thrm:1} and Theorem~\ref{thrm:2} we will prove that the distribution of cyclic(abelian, nilpotent) numbers in Beatty sequence differs from the original result only by a factor $\alpha^{-1}$, that is: $\#C^*(x)\sim\frac{C(x)}{\alpha}$, $\#A^*(x)\sim\frac{A(x)}{\alpha}$, $\#N^*(x)\sim\frac{N(x)}{\alpha}$. This is consistent with other arithmetical problems in combination with Beatty sequences, as studied in \cite{qi2025kfold} and \cite{wannes2024biases}. Our proof process runs in parallel to \cite{Pollack2021} and \cite{Just2023}, considering only the modification of some estimates to the remainder resulting within the Beatty sequence. Thanks to the fact that the error term $O_N(\frac{x}{(\log_3{x})^N})$ for any fixed $N>0$ from the original asymptotic result is no larger than the error terms resulting after the introduction of $\mathcal{B}_{\alpha,\beta}$, we do not need significant modifications to the original proof.

\section{Notations}
\ We define 
\begin{align*}
    e(t):=e^{2\pi it}, \ \ \ \  \left\{ t\right\}=t-[t].
\end{align*}
\par The sawtooth function is defined by
\begin{align*}
    \psi(t):=t-[t]-\frac{1}{2}=\left\{ t\right\}-\frac{1}{2}\ \  (t \in \mathbb{R})
\end{align*}
\par By Lemma 3.3 in \cite{banks2009character}, we know that $\alpha$, $\alpha^{-1}$, $b\alpha$ have the same type for any irrational number $\alpha$ of finite type and every integer $b \neq 0$, and by Dirichlet Theorem we know that $\tau \ge1$. Let $\epsilon > 0$ be a sufficiently small real number throughout this paper.
\par The implied constants in symbols $O$, $\ll$ and $\gg$ may depend on the parameters $\alpha$, $\beta$ and $\epsilon$ but are absolute otherwise. We recall that for functions F and G the notations $F \ll G$, $G \gg F$ and $F=O(G)$ are all equivalent to the statement that the inequality $|F| \le C|G|$ holds for some constant $C > 0$.

\section{Main results}
\par The main results may now be enunciated as follows:
\begin{thrm}\label{thrm:1}
    There is a sequence of real numbers $b_0=1,b_1,b_2,b_3,...$ such that, for each fixed positive integer $N$ and all large $x$, 
\begin{align*}
    \#C^*(x)=\frac{e^{-\gamma}x}{\alpha\log_3{x}} \left(\sum_{k=0}^N\frac{b_k}{(\log_3{x})^k} \right)+O_N\left( \frac{x}{(\log_3{x})^{N+2}}\right).
\end{align*}
\end{thrm}

\par As we mentioned, Theorem~\ref{thrm:1} is analogous to \cite{Pollack2021}. 

\begin{thrm}\label{thrm:2} 
    There are sequences of real numbers $c_0=1,c_1,c_2,c_3,\dots$ and $d_0=1,d_1,d_2,d_3,\dots$ such that, for each fixed positive integer $N$ and all large $x$, 
\begin{align*}
    \#A^*(x) - \#C^*(x) &= \frac{e^{-\gamma}x}{\alpha\log_2{x}(\log_3{x})^2} 
                       \left( \sum_{k=0}^N\frac{c_k}{(\log_3{x})^k} \right)
                       + O_N\left( \frac{x}{\log_2{x}(\log_3{x})^{N+3}}\right),\\
    \#N^*(x) - \#A^*(x) &= \frac{e^{-\gamma}x}{\alpha(\log_2{x})^2(\log_3{x})^2} 
                       \left( \sum_{k=0}^N\frac{d_k}{(\log_3{x})^k} \right)
                       + O_N\left( \frac{x}{\log_2{x}(\log_3{x})^{N+3}}\right).
\end{align*}
\end{thrm}
Theorem ~\ref{thrm:2} is analogous to \cite{Just2023}. Because the proofs of the two equalities are almost identical, for Theorem~\ref{thrm:2} we only need to prove the first one.

\section{Lemmas and Preliminaries}
\begin{lmm}\label{lmm:1} 
   Let $u_1,\dots, u_n \in \mathbb{R}$. Then for any $J \in \mathbb{N}$ and any $\rho \le \sigma \le \rho+1$, we have
\begin{align*}
\big| \#\{1\leq &n\leq N: u_n\in [\rho,\sigma]\bmod 1\} - (\sigma-\rho)N \big| \\
&\leq \frac{N}{J+1} + 3\sum_{j=1}^J \frac{1}{j}\left|\sum_{m=1}^N e(ju_m)\right|.
\end{align*}
\end{lmm}
\begin{proof} This is the Erd\H{o}s--Tur\'an inequality, see \cite{montgomery1994ten}.\end{proof}

\begin{lmm}\label{lmm:2}
    For a multiplicative function $f$ with $|f(n)|\le1$ for all positive integers $n$, assume $\alpha>1$ is an irrational number of finite type $\tau$. Then for every integer $j\in [1,N^{1/3\tau}/(\log N)^{3+3/2\tau}]$, 
    \begin{align*}
        \sum_{ m \le N  }f(m)e^{2\pi imj/\alpha}\ll \frac{N}{\log N},
    \end{align*}
    where the implied constant depends only on $\alpha$.
\end{lmm}
\begin{proof} This follows directly from the proof of Theorem 1 in \cite{GAN2008} by taking \(A = 1\), $R=(\log N)^3$, and $K=N^{1/3\tau}/(\log N)^{3+3/2\tau}$.\end{proof}

\begin{lmm}\label{lmm:3} 
    There is an absolute constant \( c \) such that, for all \( X \geq 3 \),

\[
\sum_{p \leq X} \frac{1}{p} = \log_2 X + c + O\left(e^{-K_1 \sqrt{\log X}} \right).
\]

Moreover, for all \( X \geq 3 \),

\[
\prod_{p \leq X} \left(1 - \frac{1}{p}\right) = \frac{e^{-\gamma}}{\log X} \left(1 + O(e^{-K_2 \sqrt{\log X}})\right),
\]
where $K_1, \ K_2$ are positive constants.
\end{lmm}
\begin{proof} This is the Mertens' theorem with the classical error estimate by de la Vall\'ee Poussin.\end{proof}

\begin{lmm}\label{lmm:4} 
    Let $\alpha,x >0$ and $N\ge1$. If $\alpha$ is of finite type $\tau<\infty$, then
\begin{align*}
    \sum_{1\le n\le N}\min\left\{ \frac{x}{n} , \frac{1}{\|\alpha n\|}\right\}\ll N^{1+\epsilon}+ x^{1-\frac{1}{1+\tau }+\epsilon},
\end{align*}
and
\begin{align*}
    \sum_{1\le n\le N}\min\left\{ x , \frac{1}{\|\alpha n+\beta \|}\right\}\ll N^{1+\epsilon}+ (xN)^{1-\frac{1}{1+\tau }+\epsilon}+x.
\end{align*}
\end{lmm}
\begin{proof} See \cite{qi2025kfold}, lemma 2.4.\end{proof}

\begin{lmm}\label{lmm:5} 
    For any $H \ge 1,$ there exists numbers $a_h,b_h$ such that 
\begin{align*}
    \left| \psi(t) - \sum_{0 < |h| \le H} a_h e(t h) \right| \le \sum_{\left|h\right|\le H} b_he( th),\  a_h\ll \frac{1}{\left|h\right|},b_h\ll\frac{1}  {H}  .
\end{align*}
\end{lmm}
\begin{proof} It's the Vaaler’s approximation, see \cite{vaaler1985extremal}.\end{proof}

\begin{lmm}\label{lmm:6}
    Suppose $d \in \mathbb{Z}_{\ge 1}$ is given. Then 
\begin{align*}
    \# \left\{ n\le N: d|n,\ n\in \mathcal{B}_{\alpha,\beta} \right\}-\frac{N}{\alpha d}=O(N^{\frac{\tau}{1+\tau}+\epsilon}), 
\end{align*}
where $\epsilon$ is sufficiently small.
\end{lmm}

\begin{proof} We will use Lemma~\ref{lmm:4} and Lemma~\ref{lmm:5} to prove this claim. 
\par It is easy to verify that $n\in \mathcal{B}_{\alpha,\beta}$ if and only if
\begin{align*}
    n>\alpha+\beta-1,\text{ and  }  \left[\frac{n+1-\beta}{\alpha} \right]-\left[\frac{n-\beta}{\alpha}\right]=1.
\end{align*}
If this condition is not satisfied, the right-hand side equals $0$. So by Lemma~\ref{lmm:5} and the definition of $\psi(t)$, we know that
\begin{align*}
    \# \left\{ n\le N: d|n,\ n\in \mathcal{B}_{\alpha,\beta} \right\}-\frac{N}{\alpha d}&=\sum_{\substack{ n \le N \\{d\mid n}}}\left(\left[\frac{n+1-\beta}{\alpha}\right]-\left[\frac{n-\beta}{\alpha}\right]-\frac{1}{\alpha}\right)+O(1)+O(\alpha+|\beta|)\\
    &=\sum_{\substack{ n \le N \\{d\mid n}}}\left(\psi \left(\frac{n-\beta}{\alpha}\right)-\psi\left(\frac{n-\beta +1}{\alpha}\right)\right)+O(1)\\
    &=S_1+O(S_2),
\end{align*}
where 
\begin{align*}
    S_1 =\sum_{\substack{ n \le N \\{d\mid n}}}\sum_{1\le |j| \le J}a_j\left(e\left(\frac{j(n-\beta)}{\alpha}\right)-e\left(\frac{j(n-\beta  +1)}{\alpha}\right)\right),
\end{align*}
\begin{align*}
    S_2 =\sum_{\substack{ n \le N \\{d\mid n}}}\sum_{|j|\le J}b_j\left(e\left(\frac{j(n-\beta)}{\alpha}\right)+e\left(\frac{j(n-\beta  +1)}{\alpha}\right)\right).
\end{align*}
\par Let $\theta(j):=e(-j\beta/\alpha)(1-e(j/\alpha))\ll1,$ noticing that $d/\alpha$ also has the type of $\tau$, by Lemma~\ref{lmm:4} we can get that
\begin{align*}
    S_1&=\sum_{1\le |j|\le J}a_j\theta(j) \left| \sum_{\substack{ n \le N \\{d\mid n}}}e\left(\frac{jn}{\alpha} \right) \right|\ll\sum_{1\le |j|\le J}\frac{1}{|j|} \left| \sum_{\substack{ n \le N \\{d\mid n}}}e\left(\frac{jn}{\alpha} \right) \right|= \sum_{1\le |j| \le J}\frac{1}{|j|} \left| \sum_{m\le \frac{N}{d}}e\left(\frac{jdm}{\alpha} \right) \right|\\
    &\ll  \sum_{j=1}^J\min \left \{ \frac{N/d}{j}, \frac{1}{\|jd/\alpha\|} \right\}                   \ll J^{1+\epsilon}+\left(\frac{N}{d}\right)^{1+\epsilon-\frac{1}{1+\tau}},
\end{align*}
and
\begin{align*}
    S_2&\ll\sum_{|j|\le J}\frac{1}{J} \left| \sum_{\substack{ n \le N \\{d\mid n}}}e\left(\frac{jn}{\alpha} \right) \right| \ll \frac{1}{J} \sum_{1 \le |j|\le J} \min\left\{ \frac{N}{d}, \frac{1}{\|jd/\alpha\|} \right\} +\frac{N}{dJ}\\
    &\ll \frac{1}{J}(N/d+J^{1+\epsilon}+(NJ)^{1+\epsilon-\frac{1}{1+\tau}}).
\end{align*}
Let $J =(N/d)^{\frac{1}{2+\epsilon}}$, if $\tau>1 $ we choose $\epsilon<\min \left\{ \tau -1, \frac{1}{1+\tau}\right\}$, and if $\tau=1$ then $\frac{\tau}{1+\tau }+\epsilon>\frac{1+\epsilon}{2+\epsilon}$ for every $\epsilon>0$. So we get that
\begin{align*}
    \# \left\{ n\le N: d|n,\ n\in \mathcal{B}_{\alpha,\beta} \right\}-\frac{N}{\alpha d}\ll( N/d) ^{\epsilon+\frac{\tau}{1+\tau}}+( N/d) ^{\frac{\epsilon +1}{2+\epsilon}}\ll N^{\epsilon+\frac{\tau}{1+\tau}}.
\end{align*} \end{proof}

\begin{lmm}\label{lmm:7}
Put $y\ll\log_2x$, and let $S_{x,y}=\left\{n\le x: p\mid n \Rightarrow p\ge y, \ p\in \mathcal{B}_{\alpha,\beta} \right\}$. Then 
\begin{align*}
    \#S_{x,y}=\frac{e^{-\gamma} x}{\alpha \log y}+O\left (\frac{x}{e^{K\sqrt{\log_3x}}} \right),
\end{align*}where $K$ is a positive constant.
\end{lmm}
\begin{proof}
\par Let
\begin{align*}
    \mathcal{P}=\prod_{p<y}p, \ \mathcal{A}_d:=\left\{ n\le x: d \mid n , n\in \mathcal{B}_{\alpha,\beta} \right\},
\end{align*}
then we can express this as 
\begin{align*}
    \#S_{x,y}=\sum_{\substack{ n\le x \\{(n,\mathcal{P})=1}\\{n \in \mathcal{B}_{\alpha,\beta}}}}1 =\sum_{\substack{ d\mid \mathcal{P} \\d\le x}} \mu(d)\cdot\#\mathcal{A}_d.
\end{align*}
By Lemma~\ref{lmm:6}, we have $\#\mathcal{A}_d-\frac{x}{\alpha d}=O(x^{\frac{\tau}{1+\tau}+\epsilon}),$ where $0<\epsilon<\frac{1}{1+\tau }$. Thus
\begin{align*}
    \#S_{x,y}&=\sum_{\substack{ d\mid \mathcal{P} \\d\le x}}\mu(d)\left(\frac{x}{\alpha d}+O(x^{\epsilon+\frac{\tau}{1+\tau}})\right)\\
    &=\frac{x}{\alpha}\sum_{d\mid \mathcal{P}}\frac{\mu(d)}{d}-\frac{x}{\alpha}\sum_{\substack{ d\mid \mathcal{P} \\d> x}}\frac{\mu(d)}{d}+O\left(\sum_{\substack{ d\mid \mathcal{P} \\d\le x}}\mu(d)x^{\epsilon+\frac{\tau}{1+\tau}}\right)\\
    &=\frac{x}{\alpha}\prod_{p<y}\left(1-\frac{1}{p}\right)-\frac{x}{\alpha}E_2+O(E_1),
\end{align*}
and by Rankin's trick, 
\begin{align*}
    |E_2|\le\sum_{\substack{ d\mid \mathcal{P} \\d> x}}\frac{1}{d}\le\sum_{d\mid \mathcal{P}}\frac{(d/x)^\lambda}{d}=x^{-\lambda}\prod_{p<y} (1+\frac{1}{p^{1-\lambda}})\le x^{-\lambda} \exp\left( \sum_{p<y}\frac{p^\lambda}{p} \right)(\forall \lambda>0).
\end{align*}
For $0<\lambda<1$, we have $p^\lambda\le 1+\lambda p^\lambda\log p$, so
\begin{align*}
    \sum_{p<y}\frac{p^\lambda}{p}\le \sum_{p<y}\frac{1}{p}+\lambda\sum_{p<y}\frac{p^\lambda \log p}{p}\le \log_2y+O(1)+O(\lambda y^\lambda\log y).
\end{align*}
Set $\lambda=\frac{1}{\log y}$, and this choice implies that
\begin{align*}
    E_2\ll (\log y) \exp\left(-\frac{\log x}{\log y}\right) .
\end{align*}
\par For $E_1$, obviously that 
\begin{align*}
    \sum_{\substack{ d\mid \mathcal{P} \\d\le x}}\mu(d)\ll \sum_{d\mid \mathcal{P}}1 = 2^{\pi (y)},
\end{align*}
so $E_1\ll 2^{\frac{y}{\log y}}x^{\frac{\tau}{1+\tau}+\epsilon}$.
\par If $y\ll \log_2 x$, then $E_1\ll 2^{\frac{\log_2 x}{\log_3 x}}x^{\frac{\tau}{1+\tau}+\epsilon} \ll {x}{e^{-K\sqrt{\log_3x}}},$ and $E_2\ll e^{-K\sqrt{\log_3x}}$.

\par Hence, for $y \ll \log_2 x$, there exists a constant $K>0$ such that
\begin{align*}
    \#S_{x,y}=\frac{x}{\alpha}\prod_{p<y}\left(1-\frac{1}{p}\right)+O\left (\frac{x}{e^{K\sqrt{\log_3x}}} \right),
\end{align*}
and by Lemma \ref{lmm:3}, we get that
\begin{align*}
    \#S_{x,y}=\frac{e^{-\gamma} x}{\alpha \log y}+O\left (\frac{x}{e^{K\sqrt{\log_3x}}} \right).
\end{align*}
\end{proof}

\section{Proof of Theorem~\ref{thrm:1} }
Analogous to the original work in \cite{Pollack2021}, define $y=\frac{\log_2x}{\log_3x},\ z=e^{\sqrt{\log_3x}}\log_2x$. Let $S_0$ be the set of $n\le x$ with no prime factor in $[2,y]$. For each positive integer $k$, let $S_k$ be the subset of $S_0$ consisting of numbers $n=p_1p_2\dots p_km$, where $p_1,p_2,\dots,p_k\in (y,z]$ are different primes, integer $m$ is free of prime factors in $[2,z]$, and has prime
factor $q \le x^{1/\log_2x}$ with  $q \equiv 1 \Mod{p_i}$ for some $1\le i \le k$. Moreover, put $S^*_0=S_0\cap\mathcal{B}_{\alpha,\beta}$, $S^*_k=S_k\cap\mathcal{B}_{\alpha,\beta}$.
We will use 
\begin{align*}
    \#S^*_0-\sum_{1\le k \le \log_3x}\#S^*_k
\end{align*}
to estimate $\#C^*(x)$. 

Easy to see that
\[
\left(S^*_0\setminus\bigcup_{k=1}^{\log_3x}S^*_k\right) \setminus C^*(x) \subseteq \left(S_0\setminus\bigcup_{k=1}^{\log_3x}S_k\right)\setminus C(x)
\text{, }
C^*(x) \setminus \left(S^*_0\setminus\bigcup_{k=1}^{\log_3x}S^*_k\right) \subseteq C(x) \setminus\left(S_0\setminus\bigcup_{k=1}^{\log_3x}S_k\right),
\]
then we know from section 3.1 of \cite{Pollack2021} that

    \[
\#C^*(x) = \# \left(S^*_0 \setminus \bigcup_{1 \leq k \leq \log_3 x} S^*_k \right)  + O\!\left( \frac{x}{e^{\sqrt{\log_3 x}}} \right).
\]

\par By Lemma \ref{lmm:7},
\[
\# S_0^* = \frac{e^{-\gamma} x}{\alpha \log y} + O\!\left( \frac{x}{e^{K\sqrt{\log_3 x}}} \right),
\]
where the main term differs from that of \(\# S_0 = \frac{e^{-\gamma} x}{\log y} + O\!\left( \frac{x}{e^{K\sqrt{\log_3 x}}} \right)\) by a factor of $\alpha^{-1}$, and the error terms coincide, see \cite{Pollack2021}. For the numbers $n\le x$ to be counted by $S^*_k$, first we can fix different primes $p_1, p_2, \dots, p_k\in(y,z]$, so $n$ can be expressed as $n=p_1p_2\dots p_km\in \mathcal{B}_{\alpha, \beta }$, where  $m$ is free of prime factors in $[2,z]$, and $m$ has a prime
factor $q \le x^{1/\log_2x}$ with $q \equiv 1 \Mod{p_i}$ for some $1\le i\le k$. In \cite{Pollack2021} $\#S_k(x)$ is given by a sum of contributions of the form:
\begin{align*}
    \#\left\{ n \in S_k : p_1 p_2 \dots p_k \mid n \right\}
    &\sim \Sigma_1(p_1, p_2, \dots, p_k) \\
    &= \frac{x}{p_1 \cdots p_k} 
       \prod_{p \leq z} \left( 1 - \frac{1}{p} \right) 
       \left( 1 - \prod_{\substack{z < q \leq x^{1/\log_2 x} \\ q \equiv 1 \ (\mathrm{mod}\ p_i) \\ \text{for some } i}} 
       \left( 1 - \frac{1}{q} \right) \right),
\end{align*}
and the subsequent estimates of the error term then complete the proof of Theorem 1.1 in \cite{Pollack2021}. Next we want to prove that
\begin{align*}
    \#\left\{n\in S^*_k: p_1p_2\dots p_k \mid n \right\} \sim \alpha^{-1}\#\left\{n\in S_k: p_1p_2\dots p_k \mid n \right\},
\end{align*} 
and the error term here is $O\!\left( \frac{x}{p_1p_2\dots p_ke^{\sqrt{K\log_3 x}}} \right)$,  then by the same way of \cite{Pollack2021} we have proved Theorem~\ref{thrm:1}.
\par Put 
\[
\begin{aligned}
k_1(m) &=
\begin{cases}
1, & \text{if } m \text{ has no prime factor } \le z, \\
0, & \text{otherwise};
\end{cases} \\[6pt]
k_2(m) &=
\begin{cases}
0, & \text{if } \exists \, q \mid m,\ q \le x^{1/\log_2 x},\ q \equiv 1 \pmod{p_i} \text{ for some } i\ (1 \le i \le k), \\
1, & \text{otherwise}.
\end{cases}
\end{aligned}
\]
Define $k(m)=k_1(m)(1-k_2(m))$, we have:
\begin{align*}
    \#\left\{n\in S_k: p_1p_2\dots p_k \mid n \right\}&=\sum_{m\le x/p_1p_2\dots p_k}k(m),
    \\ \#\left\{n\in S^*_k: p_1p_2\dots p_k \mid n \right\}&=\sum_{\substack{ m\le x/p_1p_2\dots p_k \\p_1p_2\dots p_km\in \mathcal{B}_{\alpha,\beta}}}k(m).
\end{align*}
\par For convenience, set $P = p_1p_2\cdots p_k,$ $N=\#\left\{n\in S_k: p_1p_2\dots p_k \mid n \right\}\le x/P$, then 
\begin{align*}
    \left| \sum_{m \le x/P} k_1(m) (1 - k_2(m)) e^{2\pi i m P/\alpha } \right|
    &\le \left| \sum_{m \le x/P} k_1(m) e^{2\pi i m P/\alpha } \right|+
       \left| \sum_{m \le x/P} k_1(m)k_2(m)e^{2\pi i mP/ \alpha } \right| .
\end{align*}
By Lemma \ref{lmm:2}, since both $k_1(\cdot)\text{ and }k_2(\cdot)$ are multiplicative, for $j\le P^{-1}(x/P)^{1/3\tau}/\log (x/P)^{3+3/2\tau}$, we have
\begin{align*}
   \left| \sum_{m \le x/P} k(m) e^{2\pi i jm P/\alpha } \right| \ll \frac{x/P}{\log({x/P})}.
\end{align*}
\par Now noticing that $Pm=[\alpha r+\beta]$ for some integer $r$ if and only if $0< \left\{\frac{Pm+1-\beta }{\alpha}\right\} \le 1/\alpha$, so by Lemma \ref{lmm:1}, for any $J\in \mathbb{Z}_{\ge1}$,
\begin{align*}
    \#\left\{ n \in S^*_k : p_1p_2\cdots p_k \mid n \right\}
    &= \sum_{\substack{ m \le x/P, \\ k(m) = 1, \\ 0 < \left\{ \frac{Pm + 1 - \beta}{\alpha} \right\} \le \frac{1}{\alpha} }} 1 \\[6pt]
    &= \alpha^{-1} N + O\!\left( \frac{N}{J+1} + \sum_{j=1}^J \frac{1}{j} \left| \sum_{m\le x/P} e^{2\pi ij \frac{m + (1-\beta)/P}{\alpha/P}} k(m) \right| \right)
    \\[6pt]
    &= \alpha^{-1} N + O\!\left( \frac{N}{J+1} + \sum_{j=1}^J \frac{1}{j} \left| \sum_{m\le x/P} e^{2\pi ij mP/\alpha} k(m) \right| \right).
\end{align*}
\par Let $J=[\log_2(x/P)]$, here $P<z^k\le (\log_2x)^{2\log_3x}\ll \log x$. Then the error term is
\begin{align*}
    \ll\frac{x/P}{\log_2(x/P)}+\frac{\frac{x}{P}\log_3(x/P)} {\log (x/P)}\ll \frac{x/P}{\log_2(x/P)}.
\end{align*}
\par Combining all the above, we have
\begin{align*}
    \#\left\{n\in S^*_k: p_1p_2\dots p_k \mid n \right\}=\frac{\#\left\{n\in S_k: p_1p_2\dots p_k \mid n \right\}}{\alpha}+O\left( \frac{x/p_1p_2\dots p_k}{\log_2(x/p_1p_2\dots p_k)}  \right),
\end{align*}
and $\frac{x/p_1p_2\dots p_k}{\log_2(x/p_1p_2\dots p_k)}\ll \frac{x/p_1p_2\dots p_k}{\log_2(x^{1/2})}\ll \frac{x}{p_1p_2\dots p_ke^{K\sqrt{\log_3x}}}$. 
\par In the same way as \cite{Pollack2021}, we get that  
\begin{align*}
    \#\left\{n\in S_k: p_1p_2\dots p_k \mid n \right\}=\Sigma_1(p_1, p_2, \dots, p_k)+O\left( \frac{x}{p_1p_2\dots p_k\log_2x}\right),
\end{align*}

\begin{align*}
    \#\left\{ n \in S^*_k : p_1p_2\cdots p_k \mid n \right\}
    &= \frac{\#\left\{ n \in S_k : p_1p_2\cdots p_k \mid n \right\}}{\alpha}
       + O\!\left( \frac{x}{p_1p_2\cdots p_k \, e^{K\sqrt{\log_3 x}}} \right) \\[6pt]
    &= \frac{\Sigma_1(p_1, p_2, \dots, p_k)}{\alpha}
       + O\!\left( \frac{x}{p_1p_2\cdots p_k \, e^{K\sqrt{\log_3 x}}} \right).
\end{align*}
This is consistent with formulas (5) and (6) of \cite{Pollack2021}, differing only by a factor of \(\alpha^{-1}\). Hence the subsequent estimates in \cite{Pollack2021} continue to hold after considering the Beatty sequence. Therefore, by the same way as \cite{Pollack2021}, we have obtained a proof of Theorem \ref{thrm:1} parallel to it.

\section{Proof of Theorem~\ref{thrm:2}}
\par The proof of Theorem~\ref{thrm:2} follows the same reasoning as Theorem~\ref{thrm:1} and is based on the argument in \cite{Just2023}. As in \cite{Just2023}, again we let $y=\frac{\log_{2}x}{\log_{3}x},\ z=e^{\sqrt{\log_3x}}\log_2x$, and $p\in(y,z]$ be a prime. Define $S(x;p,0)$ to be the set of $n\le x$ such that $n = mp^2$, where $m$ is squarefree with all its prime factors exceeding $y$, and $p \nmid m$.
Define $S(x;p,k)$ to be the subset of $S(x;p,0)$ with
the further restriction that for $n = mp^2$ in $S(x;p,k)$, there are exactly $k$ primes in the interval $(y,z]$ dividing $m$, and at least one of these primes $q$ has the property that there is a prime $r \mid m$ with $z < r \le x^{1/\log_2x}$ and $r \equiv 1 \pmod q$ or $r \equiv 1 \pmod p$. Furthermore, define sets $S^*(x;p,k)=S(x;p,k)\cap\mathcal{B}_{\alpha,\beta}$, $S^*(x;p,0)=S(x;p,0)\cap\mathcal{B}_{\alpha,\beta}$. We will use
\begin{align*}
    B^*(x)=\sum_{y<p\le z} \left( \#S^*(x;p,0)-\sum_{1\le k \le \log _3x}\ \#S^*(x;p,k)\right)
\end{align*}
to estimate $\#A^*(x) - \#C^*(x)$. As a comparison, recall that in \cite{Just2023}, \begin{align*}
    B(x)=\sum_{y<p\le z} \left( \#S(x;p,0)-\sum_{1\le k \le \log _3x}\ \#S(x;p,k)\right).
\end{align*} 
It's easy to see that numbers counted by $\#A^*(x) - \#C^*(x)$ but not $ B^*(x)$ are less than numbers counted by $A(x) - C(x)$ but not $ B(x)$, and the same holds for the reverse inclusion. Then we know from section 3 of \cite{Just2023} that
\begin{align*}
    \#A^*(x) - \#C^*(x)-B^*(x)\ll A(x)-C(x)-B(x)\ll x/z.
\end{align*}
\par For $\#S^*(x;p,k)$, put
\[
\begin{aligned}
g(m) &=
\begin{cases}
1, & \text{if } m \text{ has no prime factor } \le y,\text{ $m$ is square-free } \text{ and } p\nmid m,   \\
0, & \text{otherwise};
\end{cases}
\end{aligned}
\]
then $g(m)$ is multiplicative. By Lemma~\ref{lmm:1}, we have:
\begin{align*}
    \#S^*(x;p,0)=\sum_{\substack{ m \le x/p^2 \\ mp^{2}\in \mathcal{B}_{\alpha,\beta}}}g(m) = \alpha^{-1}\#S(x,p,0)+O\left( \frac{\#S(x,p,0)}{J+1}+\sum_{j=1}^{J}\frac{1}{j}\left|\sum_{m=1}^{x/p^2}e^{2\pi i \frac{mp^2}{\alpha}}g(m)\right|\right),
\end{align*}
here $\#S(x,p,0)=\sum_{m\le x/p^2}g(m)$, and we choose $J=[\log (x/p^2)]$. By Lemma~\ref{lmm:2}, the error term is
\begin{align*}
    \ll \frac{x\log_2\frac{x}{p^2}}{p^2\log\frac{x}{p^2}}\ll\frac{1}{e^{K\sqrt{\log_{3}x}}}
\end{align*}
for any positive constant $K>0$, so from \cite{Just2023} we get that
\begin{align*}
    \#S^*(x;p,0)=
\frac{xe^{-\gamma}}{\alpha \cdot p^2 \log y} + O\left(\frac{x}{p^2 \exp(K\sqrt{\log_3 x})}\right).
\end{align*}
\par For $\#S^*(x;p,k)$, as in the proof of Theorem~\ref{thrm:1}, we only need to show that
\begin{align*}
    \#\left\{n\in S^*(x;p,k): q_0^2q_1q_2\dots q_k \mid n \right\} \sim \alpha^{-1}\#\left\{n\in S(x;p,k): q_0^2q_1q_2\dots q_k \mid n \right\}, 
\end{align*} and the error term here is $O\!\left( \frac{x}{q_0^2q_1q_2\dots q_ke^{\sqrt{K\log_3 x}}} \right)$, where $q_0=p$, and $q_1, q_2, \dots, q_k$ are fixed increasing distinct primes in $(y,z]$. 
\par According to the argument in \cite{Just2023}, any number $n \equiv 0 \pmod{q_0^2 q_1 q_2 \dots q_k}$ counted by $S(x;p,k)$ admits a representation $n = m q_0^2 q_1 q_2 \dots q_k$, where $m$ satisfies certain properties (see the beginning of this section). Let $k(m)$ be as defined in Section 5, for convenience we put $P'=q_0^2q_1q_2\dots q_k$, $N'=\#\left\{n\in S(x;p,k): q_0^2q_1q_2\dots q_k \mid n \right\}$. Then by Lemma~\ref{lmm:1}, for any $J>0$,
\begin{align*}
     \#\left\{n\in S^*(x;p,k): q_0^2q_1q_2\dots q_k \mid n \right\}
    &= \sum_{\substack{ m \le x/P', \\ k(m) = 1, \\ 0 < \left\{ \frac{P'm + 1 - \beta}{\alpha} \right\} \le \frac{1}{\alpha} }} 1 \\[6pt]
    &= \alpha^{-1} N' + O\!\left( \frac{N'}{J+1} + \sum_{j=1}^J \frac{1}{j} \left| \sum_{m=1}^{N'} e^{2\pi i \frac{m + (1-\beta)/P'}{\alpha/P'}} k(m) \right| \right).
\end{align*}
Just as in the proof of Theorem~\ref{thrm:1}, we have 
\begin{align*}
    \#\left\{n\in S^*(x;p,k): q_0^2q_1q_2\dots q_k \mid n \right\} 
    &-\alpha^{-1}\#\left\{n\in S(x;p,k): q_0^2q_1q_2\dots q_k \mid n \right\} \\
    &\ll x\log_2(x/P')/P' \ll e^{-K\sqrt{\log_3x}}x/p^2q_1q_2\dots q_k.
\end{align*}
Summing over all terms of this type, and noticing that they differ from the original form in \cite{Just2023} only by a factor $\alpha^{-1}$ in the main term while sharing the same error term, the remainder of the argument follows exactly as in \cite{Just2023}. This completes the proof of the first equality of Theorem~\ref{thrm:2}. 
The second equality follows in exactly the same way, thereby completing the proof of Theorem~\ref{thrm:2}.

\section*{Acknowledgments}
The authors are sincerely grateful to the anonymous reviewers for their meticulous review, insightful comments, and constructive suggestions, which have greatly helped in improving the quality and clarity of this work.

\bibliographystyle{ieeetr}
\bibliography{reference}
\end{document}